\documentclass{amsart}

\usepackage{comment,amsmath,amssymb,amsthm,changepage,pifont,graphicx,tikz,soul,xcolor,longtable,todonotes,hyperref, listings, algorithm, algpseudocode, diagbox}

\usetikzlibrary{decorations.pathreplacing}

\usepackage[inline]{enumitem}

\newtheorem{lemma}{Lemma}
\newtheorem{proposition}{Proposition}
\newtheorem{corollary}{Corollary}
\newtheorem{theorem}{Theorem}
\theoremstyle{definition}

\newtheorem{problem}{Problem}
\newtheorem{example}{Example}
\theoremstyle{remark}
\newtheorem{remark}{Remark}

\DeclareMathOperator{\charac}{char}
\DeclareMathOperator{\disc}{disc}

\DeclareMathOperator{\GL}{GL}

\DeclareMathOperator{\End}{End}

\DeclareMathOperator{\cl}{cl}
\DeclareMathOperator{\tr}{tr}
\DeclareMathOperator{\res}{res}
\DeclareMathOperator{\val}{val}

\newcommand{\Ell}{\mathcal{E}\hspace{-0.065cm}\ell\hspace{-0.035cm}\ell}

\newcommand{\Z}{\mathbb{Z}}
\newcommand{\Q}{\mathbb{Q}}

\newcommand{\C}{\mathbb{C}}
\newcommand{\OO}{\mathcal{O}}
\newcommand{\FF}{\mathcal{F}}

\newcommand{\fraka}{\mathfrak{a}}
\newcommand{\frakb}{\mathfrak{b}}
\newcommand{\frakc}{\mathfrak{c}}

\newcommand{\F}{\mathbb{F}}

\mathchardef\mhyphen="2D
\newcommand{\mathsc}[1]{{\normalfont\textsc{#1}}}
\newcommand{\probname}[1]{{\mathsc{#1}}}
\newcommand{\EndRing}{\probname{EndRing}}

\newcommand{\OEndRing}[1][\OO]{{#1}\mhyphen\probname{EndRing}} 
\newcommand{\strongOEndRing}[1][\OO]{\OEndRing[{#1}]^*} 

\newcommand{\strongOTransform}[1][\OO]{\probname{Effective}\text{ }\allowbreak{#1}\mhyphen\probname{Vectorization}}
 
\newcommand{\strongOUber}[1][\OO]{\probname{Effective}\text{ }\allowbreak{#1}\mhyphen\probname{Uber}}

\newcommand{\SSO}[1][\OO]{\Ell_{\OO}(k)}

\title[On DDH for class group actions on oriented elliptic curves]{On the decisional Diffie--Hellman problem for class group actions on oriented elliptic curves}

\author[W.\ Castryck]{Wouter Castryck}
\address{imec-COSIC, KU Leuven, Kasteelpark Arenberg 10/2452, 3001 Leuven, Belgium \vspace{-0.25cm}}
\address{Dept.\ Mathematics: Algebra and Geometry, Ghent University, Krijgslaan 281, 9000 Gent, Belgium}
\email{wouter.castryck@esat.kuleuven.be}

\author[M.\ Houben]{Marc Houben}
\address{imec-COSIC, KU Leuven, Kasteelpark Arenberg 10/2452, 3001 Leuven, Belgium \vspace{-0.25cm}}
\address{Dept.\ Mathematics, KU Leuven, Celestijnenlaan 200B, 3001 Leuven, Belgium \vspace{-0.25cm}}
\address{Dept.\ Mathematics, Leiden Univ., Niels Bohrweg 1, 2333 CA Leiden, The Netherlands}
\email{marc.houben@kuleuven.be}

\author[F.\ Vercauteren]{Frederik Vercauteren}
\address{imec-COSIC, KU Leuven, Kasteelpark Arenberg 10/2452, 3001 Leuven, Belgium}
\email{frederik.vercauteren@esat.kuleuven.be}

\author[B.\ Wesolowski]{Benjamin Wesolowski}
\address{Univ.\ Bordeaux, CNRS, Bordeaux INP, IMB, UMR 5251, F-33400 Talence, France \vspace{-0.25cm}}
\address{INRIA, IMB, UMR 5251, F-33400 Talence, France}
\email{benjamin.wesolowski@math.u-bordeaux.fr}

\begin{document}

\maketitle

\begin{abstract}
   We show how the Weil pairing can be used to evaluate the assigned characters of an imaginary quadratic order $\OO$ in an unknown ideal class $[\fraka] \in \cl(\OO)$ that connects two given $\OO$-oriented elliptic curves $(E, \iota)$ and $(E', \iota') = [\fraka](E, \iota)$. 
When specialized to ordinary elliptic curves over finite fields, our method is conceptually simpler and often somewhat faster than a recent approach due to Castryck, Sot\'akov\'a and Vercauteren, who rely on the Tate pairing instead.   
   The main implication of our work is that it breaks the decisional Diffie--Hellman problem for practically all oriented elliptic curves that are acted upon by an even-order class group.
  It can also be used to better handle the worst cases in Wesolowski's recent reduction from the vectorization problem for oriented elliptic curves to the endomorphism ring problem, leading to a method that always works in sub-exponential time.
\end{abstract}

\section{Introduction}
This paper is primarily concerned with the \textsc{Decisional Diffie--Hellman problem} (DDH) for ideal class groups acting on oriented elliptic curves through isogenies. In order to state this problem precisely, we fix an
order $\OO$ in an imaginary quadratic number field $K$ along with an algebraically closed field $k$. A (primitive) \emph{$\OO$-orientation} on an elliptic curve $E$ over $k$ is an injective ring homomorphism $\iota : \OO \hookrightarrow \End(E)$ that cannot be extended to a superorder $\OO' \supsetneq \OO$ in $K$.
The set
\[ \Ell_{\OO}(k) = \{ \, (E, \iota) \, | \, \text{$E$ an elliptic curve over $k$ and $\iota$ an $\OO$-orientation on $E$} \,  \} / \cong, \]
if non-empty,
comes equipped with a free action
\begin{equation} \label{eq:action} 
\cl(\OO) \times \Ell_{\OO}(k) \longrightarrow \Ell_{\OO}(k) : ([\fraka], (E, \iota)) \longmapsto [\fraka](E, \iota)
\end{equation}
by the ideal class group  of $\OO$, see Section~\ref{sec:action} for details (including what it means for two $\OO$-oriented elliptic curves $(E, \iota)$ and $(E',\iota')$ to be isomorphic). 
Now assume that a party, say Eve, has unlimited access to samples from $\Ell_\OO(k)^3$ that are consistently of either of the following two forms:
\[ \begin{array}{lll} \big( \,  [\fraka](E, \iota), \, [\frakb](E, \iota), \, [\fraka][\frakb](E, \iota) \, \big) & & [\fraka], [\frakb] \stackrel{\$}{\leftarrow} \cl(\OO), \\
\big( \,  [\fraka](E, \iota), \, [\frakb](E, \iota), \, [\frakc](E, \iota) \, \big) & &
[\fraka], [\frakb], [\frakc] \stackrel{\$}{\leftarrow} \cl(\OO), \\
\end{array} \]
for some fixed and publicly known $(E, \iota)$.
Then Eve successfully solves DDH if she can guess, with non-negligible advantage, 
from which of these two distributions her triples were sampled.

The hardness of the decisional Diffie--Hellman problem is a natural security foundation for cryptographic constructions based on ideal class group actions, which trace back to the works of Couveignes~\cite{couveignes} and Rostovtsev--Stolbunov~\cite{rostovtsevstolbunov,stolphd} 
and
which have attracted much attention lately, in the context of post-quantum cryptography.
Here, one lets $k$ be an algebraic closure of a finite field, 
in which case all curves in $\Ell_\OO(k)$ can be defined over a common finite subfield $F \subseteq k$.
While the initial focus was on ordinary elliptic curves, whose orientations $\iota$ are just ring isomorphisms, most of the latest work is concerned with supersingular elliptic curves,
whose endomorphism rings are orders in a quaternion algebra and therefore leave room for a wide range of orientations. Here, we highlight supersingular elliptic curves defined over a finite prime field $\F_p$, which are naturally oriented by an order in $\Q(\sqrt{-p})$. The corresponding ideal class group actions underpin CSIDH~\cite{csidh} and spin-offs such as~\cite{csifish,threshold,bonehkoganwoo,OT}, and tend to yield more practical cryptosystems than in the ordinary case. 
More generally oriented supersingular elliptic curves made their first cryptographic appearance in the OSIDH protocol due to Col\`o and Kohel~\cite{osidh}. To date, this protocol remains largely theoretical, but it has attracted a good amount of recent interest, see e.g.~\cite{dartoisdefeo,onuki,Wes22a}. 

Our paper revisits the recent work~\cite{breakDDH}, which presents an efficient solution to DDH for essentially all ordinary elliptic curves over finite fields whose endomorphism ring has an even class number.
In more detail, as soon as there exists a non-trivial assigned character 
$\chi : \cl(\OO) \to \{ \pm 1 \}$ of sufficiently small modulus $m$, the attack from~\cite{breakDDH} allows Eve to compute
$\chi([\fraka])$ merely from the knowledge of
$(E, \iota)$ and $(E', \iota') = [\fraka](E, \iota)$, i.e., without knowing $[\fraka]$ itself. This 
indeed suffices to break DDH, since it allows her to check whether
$\chi([\frakc]) = \chi([\fraka])\chi([\frakb])$, which is true for $[\frakc] = [\fraka][\frakb]$, but for uniformly random $[\frakc]$ it fails with probability $1/2$. 

Unfortunately, the method from~\cite{breakDDH} is specific to ordinary curves: the attack proceeds by extending the base field and navigating to the floors of the $m$-isogeny volcanoes\footnote{Or rather $2$-isogeny volcanoes in case $m \in \{4, 8\}$.} of $(E, \iota)$ and $(E, \iota')$, with the goal of enforcing non-trivial cyclic rational $m^\infty$-torsion, and then recovering the character value using two Tate pairing computations. Beyond ordinary curves, it is generally impossible to turn the rational $m^\infty$-torsion cyclic using an isogeny walk, so this strategy fails. For supersingular elliptic curves over $\F_p$ with $p \equiv 1 \bmod 4$ equipped with their natural $\Z[\sqrt{-p}]$-orientation, where it suffices to consider the assigned character of modulus $m = 4$, an ad-hoc fix was given in~\cite[Thm.\,10]{breakDDH}, but it is unclear how this fix would generalize.

\subsection*{Contribution}
We give an alternative method for computing assigned character values $\chi([\fraka])$ purely from $(E, \iota)$ and $(E', \iota') = [\fraka](E, \iota)$, using the Weil pairing rather than the Tate pairing. Our approach deals with arbitrary orientations and works over arbitrary fields.
Moreover, it simplifies and often speeds up the attack from~\cite{breakDDH} in the case of ordinary elliptic curves over finite fields, as it avoids the need for navigating through isogeny volcanoes.
It also naturally incorporates the previously ad-hoc case of supersingular elliptic curves over prime fields.

The main result is easy enough to be stated right away; we recall that for an odd prime divisor $m \mid \disc(\OO)$, the assigned character of modulus $m$ is defined as
\begin{equation} \label{eq:oddprimechar} 
\chi_m : \cl(\OO) \to \{ \pm 1 \} : [\fraka] \mapsto \left( \frac{N(\fraka)}{m} \right) 
\end{equation}
where it is assumed that $[\fraka]$ is represented by an ideal $\fraka$ of norm coprime to $m$ (see our conventions further down) and $\left( \frac{\cdot}{m} \right)$ is the Legendre symbol.

\begin{theorem}\label{thm:mainthm}
Let $\OO$ be an imaginary quadratic order and let $(E, \iota), (E', \iota')$ be $\OO$-oriented elliptic curves connected by an ideal class $[\fraka] \in \cl(\OO)$. Let $m \mid \disc(\OO)$ be an odd prime divisor different from $\charac k$ and consider the assigned character $\chi_m : \cl(\OO) \to \{ \pm 1\}$ of modulus $m$.
Then $\OO$ admits a generator $\sigma$ (i.e.\ $\OO = \Z[\sigma]$) of norm coprime to $m$, and for any such $\sigma$ there exist points $P \in E[m]$, $P' \in E'[m]$ such that $\iota(\sigma)(P)$ is not a multiple of $P$, and likewise for $P'$. 
Moreover 
\[
 \chi_m([\fraka]) =  \left(\frac{a}{m} \right)
\]
with $a  = \log_{e_m(P,\iota(\sigma)(P))}e_m(P',\iota'(\sigma)(P'))$, regardless of the choice of such $\sigma, P, P'$.
\end{theorem}
\noindent The condition that $\sigma$ be a generator of $\OO$ can be relaxed to $\sigma \in \OO \setminus ( \Z + m\OO)$. A proof of Theorem~\ref{thm:mainthm}, along with its adaptations covering assigned characters with even modulus, can be found in Section~\ref{sec:weil_attack}.
Since these results apply to arbitrary fields, they may be of independent theoretical interest.

\subsection*{Applications and implications}
From a cryptographic viewpoint, the most important consequence is that DDH should be considered broken by classical computers for essentially all elliptic curves over finite fields that are oriented by an imaginary quadratic order $\OO$ with even class number; see Section~\ref{sec:compl} for a more in-depth discussion. 

As a more surprising application, we prove in Section~\ref{sec:applicationreduction} that the new method allows to significantly improve reductions between computational problems underlying isogeny-based cryptography. On one hand, we have the problem of computing endomorphism rings of supersingular elliptic curves. It is of foundational importance to the field, as its presumed hardness is necessary for the security of essentially all isogeny-based cryptosystems~\cite{gpst,CPV20,FKM21}. Oriented versions of this \textsc{Endomorphism Ring Problem} were introduced in~\cite{Wes22a}. On the other hand, many cryptosystems relate directly to the presumably hard inversion problem for the action of the class group $\cl(\OO)$ on oriented supersingular curves: the \textsc{Vectorization Problem}. It was proved in~\cite{Wes22a} that the vectorization problem reduces to the endomorphism ring problem in polynomial time in the length of the instance and in $\# (\cl(\OO)[2])$. Unfortunately, the dependence on $\# (\cl(\OO)[2])$ means that the reduction is, in the worst case, exponential in the size of the input, since $\#(\cl(\OO)[2])$ could be as large as $D^{1/\log\log D}$, where $D = |\disc(\OO)|$. We improve this result, by proving in Section~\ref{sec:applicationreduction} that there is a reduction from the vectorization problem to the endomorphism ring problem that, in the worst case, is sub-exponential in the length of the input.

\subsection*{Conventions}
Throughout, all ideal classes $[\fraka] \in \cl(\OO)$ 
are assumed to be represented by an ideal $\fraka$ of norm coprime to $p \disc(\OO)$, where $p = \max \{ 1, \charac k \}$. Such a representative always exists, see e.g.~\cite[Cor.\,7.17]{cox}.
For an $\OO$-oriented elliptic curve $(E, \iota)$ and a point $P \in E$, we will sometimes write $\sigma(P)$ instead of $\iota(\sigma)(P)$ if $\iota$ is clear from the context. Likewise, for $[\fraka] \in \cl(\OO)$ we will sometimes write $[\fraka]E$ for the first component of $[\fraka](E, \iota)$.

\subsection*{Paper organization} Section~\ref{sec:action} provides background: it gives the full list of assigned characters of an imaginary quadratic order and it recalls how its ideal class group acts on oriented elliptic curves. Our main Section~\ref{sec:weil_attack} contains a proof of Theorem~\ref{thm:mainthm}, as well as statements and proofs for the even-modulus counterparts. Section~\ref{sec:compl} discusses the algorithmic aspects of these results, along with their implications for the decisional Diffie--Hellman problem. Finally, in Section~\ref{sec:applicationreduction} we present our improved reduction from 
the vectorization problem for oriented elliptic curves to the endomorphism ring problem.

\subsection*{Acknowledgements} The first-listed and third-listed authors are supported 
by the European Research Council (ERC) under the European Union’s Horizon 2020 research 
and innovation programme (Grant agreement No.\ 101020788 -- Adv-ERC-ISOCRYPT)
and also by CyberSecurity Research Flanders with reference number VR20192203.
The second-listed author is supported by the Research Foundation -- Flanders (FWO) under a PhD Fellowship Fundamental Research.
The fourth-listed author is supported by the Agence Nationale de la Recherche under grants ANR MELODIA (ANR-20-CE40-0013) and ANR CIAO (ANR-19-CE48-0008). We thank the anonymous reviewers for several helpful comments, and Daniel J.\ Bernstein for suggesting to use Kedlaya--Umans factorization in the proof of Theorem~\ref{thm:compl}.


\section{Background} \label{sec:action}

\subsection{Assigned characters} \label{ssec:assigned}
The following is a very brief summary of the relevant parts of~\cite[I.\S3 \& II.\S7]{cox}, to which we refer for more details.
From genus theory, we know that each order $\OO$ in an imaginary quadratic field comes equipped with an explicit list of group homomorphisms $\cl(\OO) \to \{ \pm 1\}$, called the \emph{assigned characters}, whose joint kernel is $\cl(\OO)^2$. 
Writing
\[ \disc(\OO) = -2^f d = -2^f m_1^{f_1} m_2^{f_2} \cdots m_r^{f_r} \]
for distinct odd prime numbers $m_1, \ldots, m_r$ and exponents $f \geq 0$, $f_1, \ldots, f_r \geq 1$, this list consists of
\[ \begin{array}{ll} \chi_{m_1}, \ldots, \chi_{m_r} & \text{if $f = 0$}, \\ 
\chi_{m_1}, \ldots, \chi_{m_r}, \delta & \text{if $f = 2$ and $d \equiv 1 \bmod 4$}, \\
 \chi_{m_1}, \ldots, \chi_{m_r} & \text{if $f = 2$ and $d \equiv 3 \bmod 4$}, \\
\chi_{m_1}, \ldots, \chi_{m_r}, \delta\epsilon & \text{if $f = 3$ and $d \equiv 1 \bmod 4$}, \\
\chi_{m_1}, \ldots, \chi_{m_r}, \epsilon & \text{if $f = 3$ and $d \equiv 3 \bmod 4$}, \\
\chi_{m_1}, \ldots, \chi_{m_r}, \delta & \text{if $f = 4$}, \\
\chi_{m_1}, \ldots, \chi_{m_r}, \delta, \epsilon & \text{if $f \geq 5 $}. \\
 \end{array} \]
Here $\chi_{m_i}$ is defined as in~\eqref{eq:oddprimechar}
 and 
\[ \delta : \cl(\OO) \to \{\pm 1\} : [\fraka] \mapsto (-1)^{\frac{N(\fraka) - 1}{2}}, \quad  \epsilon : \cl(\OO) \to \{\pm 1\} : [\fraka] \mapsto (-1)^{\frac{N(\fraka)^2 - 1}{8}}.
\]
Observe that $\delta \epsilon$ can be described in one go as
\[ \delta \epsilon : \cl(\OO) \to \{ \pm 1 \} : [\fraka] \mapsto (-1)^{\frac{(N(\fraka) + 2)^2 - 9}{2}}.
\]
We write $\mu \in \{r, r+1, r+2\}$ for the total number of assigned characters.

Because the joint kernel is $\cl(\OO)^2$, any character of $\cl(\OO)$ whose order divides $2$ can be written as a product of pairwise distinct assigned characters. As it turns out, there is a unique non-trivial combination that produces the trivial character:
\begin{equation} \label{eq:charrel}
  \chi_{m_1}^{f_1 \bmod 2} \chi_{m_2}^{f_2 \bmod 2} \cdots \chi_{m_r}^{f_r \bmod 2} \delta^{\frac{d+1}{2} \bmod 2} \epsilon^{f \bmod 2} = 1.
\end{equation}
Therefore, by combining assigned characters we obtain $2^{\mu - 1}$ distinct characters. Necessarily, this quantity equals the cardinality of  $\cl(\OO) / \cl(\OO)^2 \cong \cl(\OO)[2]$.

\begin{example} \label{ex:supersing} For a prime number $p \equiv 1 \bmod 4$, the ring $\Z[\sqrt{-p}]$ has two assigned characters: $\delta$ and $\chi_p$. By~\eqref{eq:charrel} these are in fact equal to each other, and non-trivial. If $p \equiv 3 \bmod 4$ then $\Z[\sqrt{-p}]$ has only one assigned character, namely $\chi_p$, and it is trivial.
\end{example}

We often make reference to the \emph{modulus} $m$ of an assigned character $\chi$,
 which is an important complexity parameter for our attack. This is simply defined to be
 \[ \left\{ \begin{array}{ll} m_i & \text{if $\chi = \chi_{m_i}$,} \\
 4 & \text{if $\chi = \delta$,} \\ 8 & \text{if $\chi = \epsilon, \delta \epsilon$.} \\ \end{array} \right. \]
 Note that $\chi([\fraka]) = \chi([\fraka'])$ as soon as $N(\fraka) \equiv N(\fraka') \bmod m$. 
Typically $m$ is the smallest positive integer with this property, but not always (e.g., as in the case of $m_i = p$ in both examples above).

\subsection{Class group action}
We now recall how the ideal class group of $\OO$ acts on $\Ell_{\OO}(k)$. This is part of the theory of complex multiplication, which is classical for $k = \C$, while for $k$ an algebraic closure of a finite field
this was elaborated in~\cite[\S3.9-12]{waterhouse}; see also~\cite{onuki} for the specifics of the supersingular case. For arbitrary $k$, we refer to Milne's course notes~\cite[\S7]{milne}. 

If $\iota$ is an $\OO$-orientation on an elliptic curve $E$ over $k$, then we can linearly extend it to a map $K \hookrightarrow \End^0(E)$,
where $\End^0(E) = \End(E) \otimes_\Z \Q$ denotes the endomorphism algebra. To each isogeny $\varphi : E \to E'$ we can naturally attach an embedding
\begin{equation*} 
  \iota_\Q : K \hookrightarrow \End^0(E') : \sigma \mapsto \frac{1}{\deg \varphi} \varphi \circ \iota(\sigma) \circ \hat{\varphi},
\end{equation*}
whose restriction to the preimage $\OO'$ of $\End(E')$
is an orientation that is called the \emph{induced orientation}, denoted by $\varphi_\ast \iota$.
We are primarily interested in isogenies $\varphi$ for which $\OO' = \OO$, in which case $\varphi$ is said to be \emph{horizontal} with respect to $\iota$.
Two $\OO$-oriented elliptic curves $(E,\iota), (E', \iota')$ are called \emph{isomorphic}, denoted $(E, \iota) \cong (E', \iota')$, if there exists an isomorphism $\varphi : E \to E'$ such that $\iota' = \varphi_\ast \iota$.

The default way to construct a horizontal isogeny is by considering
an invertible ideal $\fraka \subseteq \OO$ of norm coprime to $\max \{ 1, \charac k \}$ and attaching to it the finite subgroup
\[ E[\fraka] = \bigcap_{\alpha \in \fraka} \ker \iota(\alpha). \]
Then the separable degree-$N(\fraka)$ isogeny $\varphi_\fraka : E \to E'$ with kernel $E[\fraka]$ is horizontal. In particular $E'$ comes naturally equipped with an $\OO$-orientation $\iota' = \varphi_{\fraka \ast} \iota$. 
The pair $(E', \iota')$
is well-defined up to isomorphism and only depends on the class of $\fraka$ inside $\cl(\OO)$; we write $[\fraka](E, \iota) := (E', \iota')$. This defines the map~\eqref{eq:action}, which turns out to be a free group action. 

\begin{remark} 
In general the action is not transitive, where one subtlety is reflected in~\cite[Prop.\,3.3]{onuki}; see also the example in~\cite[\S3.1]{onuki} and the proof of~\cite[Thm.\,4.5]{schoof}. This has no consequences
for the current paper, since we are working in a single orbit, namely that of the starting curve $(E, \iota)$. 
\end{remark}

\section{Evaluating characters using the Weil pairing} \label{sec:weil_attack}

In this section we prove Theorem~\ref{thm:mainthm} and discuss its analogues for the assigned characters $\delta, \epsilon, \delta \epsilon$. In all cases it is assumed that $p = \max \{ 1, \charac k\}$ is coprime to the modulus of the character under consideration.
If $p$ is an odd prime then $\chi_p$, if it appears in the list of assigned characters, can be computed from the other characters using the relation~\eqref{eq:charrel}; see for instance Example~\ref{ex:supersing} where we had $\chi_p = \delta$.
If $p = 2$ then the same conclusion holds for $\delta$, $\epsilon$ or $\delta \epsilon$, because in even characteristic at most one of these three characters 
can appear in the list of assigned characters.\footnote{If $(E, \iota)$ is an $\OO$-oriented elliptic curve over an algebraically closed field $k$ with $\charac k = 2$, then $2^5 \nmid \disc(\OO)$. Indeed, 
if we would have $2^5 \mid \disc(\OO)$ then $E$ is necessarily supersingular, hence it concerns $y^2 + y = x^3$, the unique supersingular elliptic curve in characteristic $2$. Its endomorphism ring is isomorphic to the ring of Hurwitz quaternions $H$, and it is easy to check that every embedding $\OO \hookrightarrow H$ can be extended to an embedding $\OO' \hookrightarrow H$ with $\disc(\OO') = \disc(\OO)/4$.
See~\cite[Prop.\,3.2]{onuki} for a generalization of this observation.}

\subsection{Preliminaries}

\begin{lemma}\label{lem:gen1}
Let $\OO$ be an imaginary quadratic order and let $m$ be an odd prime number. Then $\OO=\Z[\sigma]$ for some $\sigma\in\OO$ of norm coprime to $m$.
\end{lemma}
\begin{proof}
Let $\tau\in\OO$ be a generator of $\OO$, suppose of norm divisible by $m$. Then for any $k\in\Z$,
\begin{equation*}
N(\tau+k)=N(\tau)+k(\tr(\tau)+k)\equiv k(\tr(\tau)+k) \bmod m.
\end{equation*}
Since $m \geq 3$ we can thus always find $k\in\Z$ such that $m\nmid N(\tau+k)$.
\end{proof}

\begin{lemma}\label{lem:gen2}
Let $\OO$ be an imaginary quadratic order of even discriminant. Then $\OO=\Z[\sigma]$ for some $\sigma\in\OO$ of odd norm.
\end{lemma}
\begin{proof}
Let $\tau\in\OO$ be a purely imaginary generator of $\OO$, e.g.\ $\tau=\sqrt{\disc(\OO)/4}$, where $\disc(\OO)$ is the discriminant of $\OO$. Then $N(\tau+1)=N(\tau)+\tr(\tau)+1=N(\tau)+1$, hence we can take $\sigma = \tau$ or $\sigma = \tau + 1$.
\end{proof}

\begin{lemma}\label{lem:eigenvector}
Let $\OO$ be an imaginary quadratic order, let $(E,\iota)$ be an $\OO$-oriented elliptic curve over $k$, let $m \neq \charac k$ be a prime number, and let $\sigma\in\OO$ be a generator. Then there exists a $P\in E[m]$ such that $\iota(\sigma)(P)$ is not a multiple of $P$. 
\end{lemma}
\begin{proof}
The endomorphism $\iota(\sigma)$ of $E$ induces an $\F_m$-linear map $E[m]\to E[m]$. Suppose to the contrary that every $P \in E[m]$ is an eigenvector. This can only happen if the map has the full $m$-torsion $E[m]$ as an eigenspace. Thus there exists $\lambda\in\Z$ such that $E[m]\subseteq\ker(\iota(\sigma-\lambda))$. 
It then follows that $\iota_{\Q}((\sigma-\lambda)/m)\in\End(E)$, and hence that $\sigma-\lambda\in m\OO$ by the fact that $\iota$ is a primitive embedding, i.e.\ it cannot be extended to a strict superorder of $\OO$. Since $\Z+m\OO\subsetneq\OO$ this contradicts the assumption that $\sigma$ generates $\OO$.
\end{proof}

\subsection{Evaluating the characters $\chi_m$}

We now prove Theorem~\ref{thm:mainthm}.

\begin{proof}[Proof of Theorem \ref{thm:mainthm}] 
The existence of $\sigma, P, P'$ follows from Lemma~\ref{lem:gen1} and Lemma~\ref{lem:eigenvector}. The endomorphism $\iota(\sigma)$ of $E$ induces an $\F_m$-linear map $E[m]\to E[m]$. Since $m\mid\disc(\OO) = \tr(\sigma)^2 - 4N(\sigma)$ and $m \nmid N(\sigma)$, its characteristic polynomial  has a nonzero double root, say $\alpha\in\F_m^{\times}$. Consequently, we can extend to a basis $P_0,P$ of $E[m]$ for which the matrix of $\sigma$ is in upper-triangular form $\left( \begin{smallmatrix}
\alpha & \beta \\
0 & \alpha
\end{smallmatrix} \right)$
for some $\beta\in\F_m^{\times}$. With respect to this basis any $Q\in E[m]$ that is not an eigenvector of $\sigma$ is of the form $Q=\lambda P_0+\mu P$ where $\mu\neq 0$. We see that
\begin{equation*}
e_m(Q,\sigma(Q))=e_m(\lambda P_0+\mu P,(\alpha\lambda+\beta\mu)P_0+\alpha\mu P)=e_m(P,\beta P_0)^{\mu^2} = e_m(P,\sigma(P))^{\mu^2},
\end{equation*}
showing that $e_m(P,\sigma(P))$ is independent of the choice of $P$, up to raising to powers that are nonzero squares modulo $m$. Then, of course, the same conclusion applies to $e_m(P', \sigma(P'))$. 

Recall our convention from the introduction, namely that we assume that the norm of $\fraka$, which equals the degree of the corresponding isogeny $\varphi=\varphi_{\fraka}:E\to E'$, is coprime to $m$. In particular, $P_0\not\in\ker\varphi$. 
By definition of the class group action, $\iota' = \varphi_\ast \iota$ satisfies
\begin{equation*}
\iota'(\sigma)(\varphi(P))=\left(\frac{1}{\deg\varphi}\varphi\iota(\sigma)\hat{\varphi}\right)(\varphi(P))=\varphi(\iota(\sigma)(P))=\beta\varphi(P_0)+\alpha\varphi(P),
\end{equation*}
showing that $\varphi(P)$ is not an eigenvector for $\iota'(\sigma)$ acting on $([\fraka]E)[m]$. So we see that $e_m(\varphi(P), \iota'(\sigma)(\varphi(P)))$ is obtained from $e_m(P', \iota'(\sigma)(P'))$
by raising it to a nonzero square mod $m$.
To conclude, we observe that
\begin{equation*}
e_m(\varphi(P),\iota'(\sigma)(\varphi(P)))=e_m(\varphi(P),\varphi(\iota(\sigma)(P)))=e_m(P,\iota(\sigma)(P))^{\deg\varphi}.\qedhere
\end{equation*}
\end{proof}

\subsection{Evaluating $\delta$, $\epsilon$ or $\delta \epsilon$}

We now present the analogues of Theorem~\ref{thm:mainthm} for the even-modulus characters $\delta$, $\epsilon$ and $\delta \epsilon$. We first focus on $\delta$, which, as we saw in Section \ref{ssec:assigned}, is an assigned character if and only if we can write $\disc(\OO)=-4\cdot d$ for some $d\equiv 0,1\bmod 4$.

\begin{proposition}\label{prop:delta}
Assume $\charac k \neq 2$. Let
$\OO$ be an imaginary quadratic order of discriminant $-4\cdot d$ where $d\equiv 0,1\bmod 4$, and let $(E, \iota)$, $(E', \iota')$ be $\OO$-oriented elliptic curves over $k$ connected by an ideal class $[\fraka] \in \cl(\OO)$. Then $\OO$ admits an odd-norm generator $\sigma$, and for any such $\sigma$ there exist points $P\in E[4]$, $P' \in E'[4]$ such that $\iota(\sigma)(2P)\neq 2P$ and $\iota'(\sigma)(2P')\neq 2P'$. Moreover
\begin{equation*} 
\delta([\fraka]) = (-1)^{\frac{a - 1}{2}},
\end{equation*}
with $a = \log_{e_4(P,\iota(\sigma)(P))}e_4(P',\iota'(\sigma)(P'))$, for any such choice of $\sigma, P, P'$.
\end{proposition}
\begin{proof}
The existence of $\sigma, P, P'$ follows from Lemma~\ref{lem:gen2} and Lemma~\ref{lem:eigenvector}. Note that the assumption on the discriminant of $\OO$ shows that the character $\delta$ indeed exists, and that this implies that $N(\sigma)\equiv 1\bmod 4$ (since the principal ideal class $[(\sigma)]$ lies in the kernel of $\delta$).  By upper-triangularizing the action of $\sigma$ on $E[2]$ as in the proof of Theorem~\ref{thm:mainthm}, we see that there exists a $P_0\in E[4]$ such that the matrix $M_{\sigma}$ of $\sigma$ acting on $E[4]$ with respect to the basis $P_0,P$ is of the form
\begin{equation*}
M_{\sigma}\equiv \begin{pmatrix}
1 & 1\\
0 & 1
\end{pmatrix} \bmod 2.
\end{equation*}
Since $N(\sigma)\equiv 1\bmod 4$ this means that $M_{\sigma}$ is of the form either $\left( \begin{smallmatrix}
\alpha &\beta\\
0 &\alpha
\end{smallmatrix} \right)$ or $ \left( \begin{smallmatrix}
\alpha &\beta\\
2 &-\alpha
\end{smallmatrix} \right)$, with $\alpha,\beta$ odd. Any $Q$ with the property that $\sigma(2Q) \neq 2Q$ is of the form $\lambda P_0+\mu P$ where $\mu$ is odd. If $M_{\sigma}$ is of the first form we get
\begin{equation*}
e_4(Q,\sigma(Q))=e_4(\lambda P_0+\mu P,(\alpha\lambda+\beta\mu)P_0+\alpha\mu P)=e_4(P,\beta P_0)^{\mu^2} = e_4(P,\sigma(P))^{\mu^2}.
\end{equation*}
If $M_{\sigma}$ is of the second form we again get
\begin{eqnarray*}
e_4(Q,\sigma(Q)) &=& e_4(\lambda P_0+\mu P,(\alpha\lambda+\beta\mu)P_0+(2\lambda-\alpha\mu)P) \\
&=& e_4(P,\beta P_0)^{\mu^2}e_4(P,P_0)^{2(\lambda\alpha\mu-\lambda^2)} = e_4(P,\sigma( P))^{\mu^2}
\end{eqnarray*}
where the last equality uses that $\lambda, \mu, \alpha$ are odd.
From $\mu^2 \equiv 1 \bmod 4$ it follows that $e_4(P,\sigma(P))$ does not depend on the choice of $P$.
Then, of course, the same is true for $e_4(P', \sigma(P'))$.

By our convention we assume that the norm of $\fraka$, and hence the degree of the corresponding isogeny $\varphi=\varphi_{\fraka}:E\to E'$, is odd.
In particular, $2P_0\not\in\ker\varphi$ and
\begin{equation*}
\iota'(\sigma)(\varphi(2P))=\left(\frac{1}{\deg\varphi}\varphi\iota(\sigma)\hat{\varphi}\right)(\varphi(2P))=\varphi(\iota(\sigma)(2P))=\varphi(2P_0)+\varphi(2P)
\end{equation*}
is different from $\varphi(2P)$. Thus we find that $e_4(P', \sigma(P'))$ equals
\begin{equation*}
 e_4(\varphi(P),\iota'(\sigma)(\varphi(P)))=e_4(\varphi(P),\varphi(\iota(\sigma)(P)))=e_4(P,\iota(\sigma)(P))^{\deg\varphi},
\end{equation*}
which concludes the proof.
\end{proof}



Next, we discuss the modulus-$8$ characters $\epsilon$ and $\delta \epsilon$. Note that by Section \ref{ssec:assigned}, we have that $\epsilon$ is an assigned character if and only if either $2^5\mid\disc(\OO)$ or $\disc(\OO)=-2^3\cdot d$ with $d\equiv 3\bmod 4$. Similarly, $\delta\epsilon$ is an assigned character if and only if either $2^5\mid\disc(\OO)$ or $\disc(\OO)=-2^3\cdot d$ with $d\equiv 1\bmod 4$.

\begin{proposition}\label{prop:epsilon}
Assume $\charac k \neq 2$, let
$\OO$ be an imaginary quadratic order of discriminant $\disc(\OO)\equiv -2^fd$ with $d$ odd and $f\geq 3$,
and consider $\OO$-oriented elliptic curves $(E, \iota)$, $(E', \iota')$ over $k$ connected by an ideal class $[\fraka] \in \cl(\OO)$. Assume that $\epsilon$, resp.\ $\delta \epsilon$, appears among the assigned characters of $\OO$.
Then $\OO$ admits an odd-norm generator $\sigma$, and for any such $\sigma$ there exist points $P\in E[8]$, $P' \in E'[8]$ such that $\iota(\sigma)(4P)\neq 4P$ and $\iota'(\sigma)(4P')\neq 4P'$. Moreover
$\epsilon([\fraka])$, resp.\ $\delta \epsilon([\fraka])$, can be computed as
\begin{equation*} 
\epsilon([\fraka]) =   (-1)^{\frac{a^2 - 1}{8}}, \ \ resp.\
 \  \delta \epsilon([\fraka]) = (-1)^{\frac{ \left( a + 2 \right)^2 - 9}{8}},
\end{equation*}
with $a = \log_{e_8(P,\iota(\sigma)(P))}e_8(P',\iota'(\sigma)(P'))$, and for any such choice of $\sigma, P, P'$.
\end{proposition}
\begin{proof}
As in the previous proof, the existence of $\sigma, P, P'$ follows from Lemma~\ref{lem:gen2} and Lemma~\ref{lem:eigenvector}. 
The main difference with the foregoing proofs is that if $Q \in E[8]$ is another point satisfying $\sigma(4Q) \neq 4Q$, then $e_8(Q, \sigma(Q))$ relates more subtly to $e_8(P, \sigma(P))$. Namely, we will argue that
\begin{equation} \label{eq:propepssuff} 
e_8(Q, \sigma(Q)) \in \left\{ e_8(P, \sigma(P)), e_8(P, \sigma(P))^{N(\sigma)} \right\}, 
\end{equation}
and then of course the same again applies to $e_8(P', \sigma(P'))$. This will then lead to the conclusion that
\[ e_8(P', \sigma(P')) \in \left\{ e_8(P, \sigma(P))^{\deg \varphi}, e_8(P, \sigma(P))^{N(\sigma) \deg \varphi} \right\}, \]
which is indeed sufficient, since the principal ideal class $[(\sigma)]$ has trivial character values. More explicitly, if $\epsilon$ exists then we must have $N(\sigma) \bmod 8 \in \{1,7\}$, while if $\delta \epsilon$ exists then we have $N(\sigma) \bmod 8 \in \{1, 3\}$.

In order to prove~\eqref{eq:propepssuff}, note that, since $N(\sigma) \equiv 1 \bmod 2$,
\begin{equation*}
\tr(\sigma)^2+4 \equiv \tr(\sigma)^2-4\cdot N(\sigma) = \disc(\OO) \equiv 0 \bmod 8,  
\end{equation*}
so that $\tr(\sigma)\equiv 2\bmod 4$. It follows that the characteristic polynomial of $\sigma$ modulo $4$ is $X^2+2X+N(\sigma)$, hence we can extend to a basis $P_0,P$ of $E[8]$ such that the matrix of $\iota(\sigma)$ acting on $E[8]$ is of the form
\begin{equation*}
M_{\sigma} \equiv \begin{cases}
\begin{pmatrix}
\alpha & \beta\\
0 & \alpha
\end{pmatrix} \bmod  4 \qquad\quad \mbox{if } N(\sigma)\equiv 1\bmod 4,\\
\begin{pmatrix}
\alpha & \beta\\
2 & \alpha
\end{pmatrix} \bmod  4 \qquad\quad \mbox{if } N(\sigma)\equiv 3\bmod 4,
\end{cases} 
\end{equation*}
with $\alpha,\beta$ odd. It follows that
\begin{equation*}
M_{\sigma}^2 \equiv \begin{cases}
\begin{pmatrix}
1 & 2\\
0 & 1
\end{pmatrix} \bmod  4 \qquad\quad \mbox{if } N(\sigma)\equiv 1\bmod 4,\\
\begin{pmatrix}
3 & 2\\
0 & 3
\end{pmatrix} \bmod  4 \qquad\quad \mbox{if } N(\sigma)\equiv 3\bmod 4.
\end{cases} 
\end{equation*}
In any case we can record that 
\begin{equation}\label{eq:PssP}
e_8(P,\sigma^2(P))^2 = e_8(P, P_0)^4 =-1.
\end{equation}
Now, with respect to the basis $P,\sigma(P)$, the matrix of $\iota(\sigma)$ acting on $E[8]$ is congruent to $\left( \begin{smallmatrix}
0 & 1 \\
1 & 0
\end{smallmatrix} \right)\bmod 2$.
Any other $Q=\lambda P+\mu\sigma(P)$ such that $\sigma(4Q)\neq 4Q$ thus has exactly one of $\lambda,\mu$ odd. We now proceed to showing~\eqref{eq:propepssuff}.
If $\mu$ is odd then we can write $\sigma(Q)=\lambda' P+\mu'\sigma(P)$ with $\lambda'$ odd, so since
\begin{equation*}
e_8(Q,\sigma(Q))^{N(\sigma)}=e_8(\sigma(Q),\sigma^2(Q))
\end{equation*}
we may reduce to the case where $\lambda$ is odd (and $\mu$ is even). For odd $\lambda$, we have
\begin{equation*}
e_8(Q,\sigma(Q))=e_8(\lambda^{-1}Q,\sigma(\lambda^{-1}Q))^{\lambda^2}=e_8(\lambda^{-1}Q,\sigma(\lambda^{-1}Q)),
\end{equation*}
hence we may further reduce to the case where $\lambda=1$. 
Now note that
\begin{eqnarray*}
e_8(P+\mu\sigma(P),\sigma(P)+\mu\sigma^2(P)) &=& e_8(P,\sigma(P))e_8(\sigma(P),\sigma^2(P))^{\mu^2}e_8(P,\sigma^2(P))^\mu\\
&=& e_8(P,\sigma(P))e_8(P,\sigma(P))^{4\frac{\mu^2}{4}N(\sigma)}e_8(P,\sigma^2(P))^{2\frac{\mu}{2}}\\
&=& e_8(P,\sigma(P))\cdot(-1)^{\frac{\mu^2}{4}}\cdot(-1)^{\frac{\mu}{2}}\\
&=& e_8(P,\sigma(P)),
\end{eqnarray*}
where in the third equality we used \eqref{eq:PssP}.
\end{proof}

\begin{remark} If $\OO$ is an imaginary quadratic order of discriminant $\disc(\OO)\equiv 0\bmod 2^5$, then
both $\epsilon$ and $\delta \epsilon$ and hence $\delta = (\delta \epsilon)\epsilon$ exist, so that $N(\sigma) \equiv 1 \bmod 8$.
In this case there is a well-defined group homomorphism $\gamma:\cl(\OO)\to(\Z/8\Z)^{\times}:[\fraka]\mapsto N(\fraka)\bmod 8$ through which $\delta, \epsilon, \delta \epsilon$ factor. This is the only situation where one can get finer-than-binary modular information about $N(\fraka)$ from $[\fraka]$; the above proof shows that we can recover $\gamma([\fraka])$ at once as 
$\log_{e_8(P,\iota(\sigma)(P))}e_8(P',\iota'(\sigma)(P'))$.
\end{remark}

\begin{remark}
In the statements of Theorem~\ref{thm:mainthm}, Proposition~\ref{prop:delta} and Proposition~\ref{prop:epsilon}, the condition that $\sigma$ be a generator of $\OO$ can in fact be relaxed to 
$\sigma \in \OO \setminus (\Z + m\OO)$ if $m$ is odd and to $\sigma \in \OO \setminus (\Z + 2\OO)$ if $m$ is even, without modifying the proofs.
\end{remark}

Wrapping up, we have given justification for Algorithm~\ref{alg:eval} below, 
evaluating an
 assigned character $\chi : \cl(\OO) \to \{ \pm 1 \}$ of modulus $m$ coprime to $\max \{ 1, \charac k\}$ in an unknown ideal class $[\fraka]$
 connecting two given $\OO$-oriented curves $(E, \iota)$ and $(E', \iota')$. 
Here, by the field of definition of $(E, \iota)$, $(E',\iota')$ we mean any (e.g., the smallest) subfield $F \subseteq k$ over which the curves $E, E'$ and the endomorphisms in $\iota(\OO), \iota'(\OO)$ are defined. 

\algrenewcommand\algorithmicrequire{\textbf{Input:}}
\algrenewcommand\algorithmicensure{\textbf{Output:}}

\begin{algorithm}
\caption{Evaluating an assigned character in an unknown ideal class}\label{alg:eval}
\begin{algorithmic}[1]
\Require 
\Statex $\OO$-oriented curves $(E, \iota)$, $(E', \iota')$ in the same orbit with field of definition $F$
\Statex an assigned character $\chi$ of $\cl(\OO)$ with modulus $m$ coprime to $\max \{ 1, \charac F\}$
\Ensure
\Statex $\chi([\fraka]) \in \{ \pm 1\}$, where $[\fraka] \in \cl(\OO)$ is such that $(E', \iota') = [\fraka](E, \iota)$  
\Statex \vspace{-0.25cm}
\State Find a generator $\sigma$ of $\OO$ of norm coprime to $m$.
\State \label{it:extdegree}Base-change to the smallest extension $\FF \supseteq F$ over which all points in $E[m]$ are defined; necessarily, then also all of $E'[m]$ is defined over $\FF$.
\State Find a point $P \in E(\FF)$ such that $E[m] = \langle P, \iota(\sigma)(P) \rangle$ and compute $\zeta = e_m(P, \iota(\sigma)(P))$.
\State Likewise, find
  a point $P' \in E'(\FF)$ such that $E'[m] = \langle P', \iota'(\sigma)(P') \rangle$ and compute $\zeta' = e_m(P', \iota'(\sigma)(P'))$.
\State \label{it:dlog}Inside $\mu_m \subseteq \FF^\times$, compute $a  =\log_\zeta \zeta'$.
\State \label{it:eval}If $m$ is an odd prime then recover $\chi([\fraka])$ as $\left( \frac{a}{m} \right)$, else recover $\chi([\fraka])$ as 
\[ (-1)^{\frac{a - 1}{2}}, \quad (-1)^{\frac{a^2 - 1}{8}}, \quad (-1)^{\frac{(a + 2)^2 - 9}{8}}, \]  
  depending on whether $\chi = \delta, \epsilon, \delta \epsilon$, respectively.
\end{algorithmic}
\end{algorithm}

\section{Complexity and consequences for DDH} \label{sec:compl}

Running Algorithm~\ref{alg:eval} in practice comes with challenges that are specific to our field of definition $F$. 
Nevertheless,  before going into a more detailed analysis of our main case of interest, namely where $F$ is a finite field, let us add some general comments to its six numbered steps:
\begin{enumerate}
  \item[1.] Very easy, by following the proof of Lemma~\ref{lem:gen1} or Lemma~\ref{lem:gen2}.
  \item[2.] The degree of $\FF / F$ is a divisor of the order of $\GL_2(\Z / m\Z)$, which is $O(m^4)$.
  \item[3.--4.] For $m$ an odd prime, the proof of Theorem~\ref{thm:mainthm} shows that the set of $m$-torsion points that are independent of their image under $\sigma$ has size $m^2 - m$. So it suffices to try $O(1)$ random points $P \in E[m]$, compute $\iota(\sigma)(P)$ and check whether $e_m(P, \iota(\sigma)(P))$ is a primitive $m$th root of unity (i.e., not $1$).\footnote{Alternatively, one may opt for a more deterministic approach by computing and analyzing a matrix of $\iota(\sigma)$ acting on $E[m]$, in which case two evaluations of $\iota(\sigma)$ will do. Note however that writing down a matrix of $\iota(\sigma)$ comes at the cost of computing some discrete logarithms.}
  \item[5.] Pollard-$\rho$ type algorithms allow us to compute the discrete logarithm using $O(\sqrt{m})$ operations in $\mu_m$.
  \item[6.] Trivial.
\end{enumerate}
The main bonus we get from working over a finite field lies in~\eqref{it:extdegree}. In this case the degree of $\FF / F$ equals the order of the Frobenius endomorphism $\pi_F$ acting on $E[m]$. While the order of $\GL_2(\Z/m\Z)$ is $O(m^4)$, the order of a single element is $O(m^2)$. 

\begin{theorem} \label{thm:compl}
Let $\OO = \Z[\sigma]$ be an imaginary quadratic order and consider two $\OO$-oriented elliptic curves $(E, \iota)$ and $(E', \iota')$ that belong to the same orbit under the action of $\cl(\OO)$, say given in Weierstrass form and connected by an unknown ideal class $[\fraka]$. Assume that $E, E', \iota(\OO), \iota'(\OO)$ are all defined over a finite field $\F_q$. Let $\chi$ be an assigned character of $\OO$ with modulus $m$ coprime to $q$.
There exists a randomized algorithm for computing $\chi([\fraka])$ that is expected to use
\begin{equation} \label{eq:compl} 
  \widetilde{O}(m^3 \log^2 q) 
\end{equation}
bit operations and $O(1)$ calls to $\iota(\sigma), \iota'(\sigma)$.
\end{theorem}

\begin{proof}
If we write $f_E(x,y)$ for the defining Weierstrass polynomial of $E$ and $\Psi_{E, m}(x)$ for its $m$-division polynomial, then the field $\FF$ can be constructed as the splitting field
of the resultant $r_{E,m}(x) = \res_y(f_E, \Psi_{E,m})$, whose degree is $O(m^2)$. The division polynomial $\Psi_{E,m}(x)$ can be computed recursively and the resultant $r_{E, m}(x)$ can be factored 
using Kedlaya--Umans~\cite{Kedlaya-Umans}. Using fast arithmetic, this takes a combined time of~\eqref{eq:compl}. Note that we obtain all points in $E[m]$ as a by-product; once we know $\FF$ we can sample points from $E'[m]$ faster.  The Weil pairings can be computed using Miller's algorithm, taking $O(\log m)$
operations in $\FF$, and Pollard-$\rho$ takes an expected $O(\sqrt{m})$ operations in $\FF$, so
these costs are dominated by~\eqref{eq:compl}, again assuming fast arithmetic. Finally, while the norm of the given generator $\sigma$ may not be coprime to $m$, from the proofs of Lemma~\ref{lem:gen1} and Lemma~\ref{lem:gen2} we see that
we can instead work with $\sigma + k$, for some positive integer $k$ bounded by $m$. Since $\iota(\sigma + k) = \iota(\sigma) + [k]$, the overhead this causes is clearly absorbed by~\eqref{eq:compl}; and similarly for $\iota'(\sigma + k)$.
\end{proof}

The effectivity of this algorithm co-depends on how easy it is to evaluate $\iota(\sigma)$ and $\iota'(\sigma)$, which is a separate discussion that is captured by the notion of \emph{efficient representations}, see Section~\ref{ssec:endringproblem} and \cite{Wes22b} for more details. One special but interesting case is where $\iota(\sigma)$ equals $\pi_{\F_q}$, or is easily derived from it, whose cost is quasi-quadratic in $m \log q$. So, in this case, the overall cost remains estimated by~\eqref{eq:compl}. This matches with the asymptotic runtime of the Tate pairing attack from~\cite{breakDDH}, as estimated in~\cite[\S5.1]{breakDDH}.\footnote{Here and below, for simplicity, the height $h \approx \val_m(\tr(\pi_\FF)^2 - 4\#\FF)$ of the $m$-isogeny volcano of $E$ over $\FF$ is estimated by $O(1)$.} 
 
While the Weil pairing attack is conceptually simpler (no descent of the isogeny volcano needed), in general one should expect the Tate pairing attack to run faster in practice. The main reason is that there it suffices to work over a field $\FF$ such that $E$ admits an $\FF$-rational point of order $m$, rather than requiring all $m$-torsion to be $\FF$-rational (in turn, this is because the Tate pairing admits non-trivial self-pairing values, in contrast with the Weil pairing). The degree of such an extension field is bounded by $O(m)$, rather than by $O(m^2)$. 
But the comparison turns in favour of the Weil pairing as soon as $E[m] \subseteq E(\F_q)$, where no field extension is needed. Note that, here, it makes more sense to measure the cost of a call to $\iota(\sigma),\iota'(\sigma)$ by the cost of evaluating $(\pi_{\F_q} - 1)/m^s$, where $s$ is maximal such that $E[m^s] \subseteq E(\F_q)$; see~\cite[Lem.\,1]{ruck}. For this we need $s$ successive point divisions by $m$; the cost of such a division is dominated by that of finding a root of a polynomial of degree $m^2$, which can be done in time
\begin{equation} \label{eq:asquad} 
  \widetilde{O}(m^2 \log^2 q),
\end{equation}
see~\cite[\S2]{rabin}. This now becomes the dominant cost of the attack. The asymptotic cost of the Tate pairing also drops to~\eqref{eq:asquad} in this case, but the Weil pairing attack comes with less overhead.

All this aside, let us re-emphasize that the Weil pairing approach works in far greater generality: for arbitrary orientations and over any field admitting explicit computation. A proof-of-concept implementation of the new method can be found at \url{https://github.com/KULeuven-COSIC/oriented_DDH}. At the time of publication, this implementation handles the case of $\Z[\sqrt{-p}]$-oriented elliptic curves in characteristic $p \equiv 1 \bmod 4$. We intend to extend the repository in due course, by also covering the higher-degree group actions that were described in~\cite{CS21}.

\subsection*{Consequences for DDH} If $\cl(\OO)$ admits a non-trivial assigned character whose modulus $m$ is sufficiently small, say polynomially bounded by $\log \disc(\OO)$, 
and if it satisfies $\gcd(m,q) = 1$, then we can use this character to
distinguish between random triples and Diffie--Hellman triples with probability $1/2$, as explained in the introduction. So, in this case, we can consider
the decisional Diffie--Hellman problem broken for $\OO$-oriented elliptic curves over $\F_q$.
More generally, if $\cl(\OO)$ admits $s \geq 1$ independent such characters (meaning that one cannot use the relation~\eqref{eq:charrel} to rewrite one of the characters in terms of the others), then we can distinguish with probability $1 - 1/2^s$.

A sufficient condition for the existence of such a character is that $\disc(\OO)$ has at least two small odd prime factors different from $p = \charac \F_q$.\footnote{In serious cryptographic applications, one can ignore the phrase ``different from $p = \charac \F_q$". Indeed, if $p \mid \disc(\OO)$ then $E$ and $E'$ are necessarily supersingular, so if moreover $p$ is small then we can compute $\End(E)$ and $\End(E')$ by navigating through all $O(p)$ nodes of the supersingular isogeny graph.
As a result, one is skating on very thin ice (see Section~\ref{sec:applicationreduction}).} Heuristically, we expect that this applies to a density $1$ subset of all imaginary quadratic orders when ordered by the absolute value of their discriminant. This can be backed up using Mertens' third theorem;  or see~\cite[III.\S6]{tenenbaum} for more dedicated tools.

As discussed in~\cite[\S6]{breakDDH}, one can thwart the attack by restricting the class-group action to $\cl(\OO)^2$, or at least to a subgroup of $\cl(\OO)$ on which all assigned characters of small modulus have trivial evaluations. However, this may have practical consequences in terms of key generation and key validation. Moreover, we do not rule out that the attack can be modified to work for characters whose order is a larger power of $2$, e.g., in view of~\cite{BosS96,stevenhagen}. Quantumly, it is known that $2^r$-torsion subgroups, for any small fixed value of $r$, do not contribute to the hardness of the vectorization problem anyway~\cite{fusion}.
Therefore, the cleanest way out is to follow the
recommendation from~\cite[\S6]{breakDDH}, namely to only work with orientations by imaginary quadratic orders whose class number is odd. There may be constructive reasons to deviate from this, e.g., as in the OSIDH protocol~\cite{osidh} where one uses orders of large prime power conductor in an imaginary quadratic field with class number one (such orders always have even class number).

\begin{remark}
It is interesting to view Theorem~\ref{thm:compl} against the \emph{classical} decisional Diffie--Hellman problem, namely for exponentiation in a group $G = \langle g \rangle$ of some large prime order $m$. Note that exponentiation defines a free and transitive action of $(\Z / m\Z)^\times$ on the set of generators of $G$.
The Legendre symbol
\[ \chi : (\Z / m\Z)^\times \to \{ \pm 1 \} : a \mapsto \left( \frac{a}{m} \right) \]
is the unique quadratic character, of modulus $m$, and if one could cook up an efficient classical way for computing $\chi(a)$
merely from the knowledge of $g$ and $g^a$, then this would 
break DDH in this setting.
This would be a spectacular result; in general,
to the best of our knowledge, we cannot do significantly better than computing $a$ using Pollard-$\rho$ and then evaluating $\chi$ at $a$. This should be compared to steps~\ref{it:dlog}.\ and~\ref{it:eval}.\ from Algorithm~\ref{alg:eval}. In other words, one could say that classical DDH is not weakened by the existence of $\chi$ because its modulus is large. 
\end{remark}

\section{Reductions to endomorphism ring computation}\label{sec:applicationreduction}

In this section, we prove that our main result Theorem~\ref{thm:mainthm} allows to significantly improve reductions between computational problems underlying isogeny-based cryptography. 
It was proved in~\cite{Wes22a} that two such families of problems are tightly connected: there are computational reductions from action inversion problems (called $\strongOTransform$ or $\strongOUber$) to endomorphism ring computation problems (called $\OEndRing$ and $\strongOEndRing$).
However, these reductions are exponential in the worst case. In this section, we apply Theorem~\ref{thm:mainthm} to obtain reductions that are sub-exponential in the worst case, and even polynomial in many regimes of interest.
All results in this section that start with (ERH), such as Theorem~\ref{thm:transformtoendring}, assume the extended Riemann hypothesis --- precisely, the Riemann hypothesis for Hecke $L$-functions.

\subsection{The supersingular endomorphism ring problem} \label{ssec:endringproblem}
In this section, we assume that the field $k$ is an algebraic closure of a finite field of characteristic $p$, and that $p$ does not split in $\OO$, nor does it divide the conductor of $\OO$. Then,  the set $\Ell_{\OO}(k)$ is non-empty and all curves in it are supersingular; this set is often denoted by $\mathsc{SS}_{\OO}(p)$ in the literature~\cite[Prop.\,3.2]{onuki}. Recall that a curve $E/k$ is supersingular if and only if its endomorphism ring $\End(E)$ is isomorphic to a maximal order in the quaternion algebra
$$B_{p,\infty} = \left(\frac{-q,-p}{\Q}\right) = \Q + \Q i + \Q j + \Q ij,$$
with the multiplication rules $i^2 = -q$, $j^2 = -p$, and $ji = -ij$, where $q$ is a positive integer that depends on $p$.

Given a supersingular elliptic curve $E$ over $k$, the endomorphism ring problem $\EndRing$ consists in computing four endomorphisms that form a basis of $\End(E)$. There is flexibility in how these endomorphisms can be represented, but we always assume that it is an \emph{efficient representation}. As in~\cite{Wes22b}, we say that an isogeny $\varphi: E\rightarrow E'$ is given in an efficient representation if there is an algorithm to evaluate $\varphi(P)$ for any $P \in E(\F_{p^r})$ in time polynomial in the length of the representation of $\varphi$ and in $r \log(p)$. We also assume that an efficient representation of $\varphi$ has length $\Omega(\log(\deg(\varphi)))$.

This endomorphism ring problem is of foundational importance to isogeny-based cryptography: it is presumed to be hard, and this hardness is \emph{necessary} (and sometimes sufficient) for the security of essentially all isogeny-based protocols~\cite{gpst,CPV20,FKM21}. It does not, however, capture well the notion of orientation, which plays an important role in many protocols. Therefore, the following oriented variants were introduced in~\cite{Wes22a}.
Computationally, an $\OO$-orientation $\iota$ is represented by a generator $\sigma$ of $\OO$ (i.e., $\OO = \Z[\sigma]$) together with an efficient representation of the endomorphism $\iota(\sigma)$.

\begin{problem}[\protect{$\OEndRing$}] Given $(E,\iota) \in \SSO$, find a basis of $\End(E)$.
\end{problem}

\begin{problem}[\protect{$\strongOEndRing$}] Given an $\OO$-orientable curve $E$, find a basis of $\End(E)$, and an $\OO$-orientation of $E$ expressed in this basis.
\end{problem}

Clearly, $\OEndRing$ reduces to $\strongOEndRing$.

\subsection{Action inversion problems}
Many cryptosystems relate, directly or more subtly, to an inversion problem for the action of $\cl(\OO)$ on $\SSO$. In essence, given $(E,\iota)$ and $(E',\iota')$ in $\SSO$, find a class $[\mathfrak a]$ such that $(E',\iota') \cong [\mathfrak a] (E,\iota)$ (or decide that it does not exist). This is called the vectorization problem. It is too weak for many practical purposes, because knowledge of the class $[\mathfrak a]$ is not sufficient to efficiently apply its action on any other $\OO$-oriented curve. Therefore, the following stronger problem was introduced in~\cite{Wes22a}.

\begin{problem}[\protect{$\strongOTransform$}]
Given three $\OO$-oriented supersingular curves $(E,\iota),(E',\iota'),(F,\jmath) \in \SSO$, find an $\OO$-ideal $\mathfrak a$ (or decide that it does not exist) such that $(E',\iota') \cong [\mathfrak a] (E,\iota)$, and an efficient representation of $\varphi_\mathfrak a: (F,\jmath) \rightarrow [\mathfrak a] (F,\jmath)$.
\end{problem}
The security of many cryptosystems directly reduces to this problem, such as CSIDH~\cite{csidh}, CSI-FiSh~\cite{csifish}, CSURF~\cite{CSURF}, or other generalizations~\cite{CS21}.\\

One can define a similar problem where no orientation is provided for $E'$. Then, one cannot require $(E',\iota') \cong [\mathfrak a] (E,\iota)$ anymore, but one can still ask for $E' \cong [\mathfrak a]E$.
The resulting \emph{Uber isogeny problem} was introduced in~\cite{Seta}.

\begin{problem}[\protect{$\strongOUber$}]
Given two $\OO$-oriented curves $(E,\iota), (F,\jmath) \in \SSO$ and an $\OO$-orientable curve $E'$, find an $\OO$-ideal $\mathfrak a$ such that $E' \cong [\mathfrak a]E$, and an efficient representation of $\varphi_\mathfrak a : (F,\jmath) \rightarrow [\mathfrak a] (F,\jmath)$.
\end{problem}

This $\strongOUber$ problem is significantly harder than the $\strongOTransform$ problem. In fact, most isogeny-based cryptosystems reduce to an instance of $\strongOUber$~\cite{Seta}, even cryptosystems such as SIDH~\cite{jao-defeo} which, at first sight, do not seem to involve any orientation.

\subsection{Action inversion reduces to endomorphism ring}
Strengthening and generalizing a result of~\cite{CPV20}, it was proved in~\cite{Wes22a} that $\strongOTransform$ reduces to $\OEndRing$, and that $\strongOUber$ reduces to $\strongOEndRing$. Both reductions are in polynomial time in the length of the instance, and in $\# (\cl(\OO)[2])$. Unfortunately, the dependence on $\# (\cl(\OO)[2])$ means that the reduction is, in the worst case, exponential in the size of the input, since $\#(\cl(\OO)[2])$ could be as large as $D^{1/\log\log D}$, where $D = |\disc(\OO)|$. The issue is the following: given two oriented curves $(E,\iota)$ and $(E',\iota')$ as in the definition of $\strongOTransform$, the reductions first find a class $[\mathfrak a]^2$ such that $(E',\iota') \cong [\mathfrak a] (E,\iota)$. Finding $[\mathfrak a]$ from $[\mathfrak a]^2$ is a square root computation. There are $\#(\cl(\OO)[2])$  square roots of $[\mathfrak a]^2$, but only one is the correct class $[\mathfrak a]$. In~\cite{Wes22a}, one simply does an exhaustive search. Now, thanks to Theorem~\ref{thm:mainthm}, there is a much more efficient way to find the correct square root, which in the worst case is sub-exponential in $\disc(\OO)$. This is the following proposition.
Recall the $L$-notation $$L_x(\alpha) = \exp\left(O\left((\log x)^\alpha(\log \log x)^{1-\alpha}\right)\right)$$ for sub-exponential complexities.

\begin{proposition}[ERH]\label{prop:findgoodroot}
Given $\OO$ of discriminant $-D$, the factorization $D = \prod_{i = 1}^{\omega(D)}\ell_i^{e_i}$ (with $\ell_i < \ell_{i+1}$),
two $\OO$-oriented elliptic curves $(E,\iota),(E',\iota')\in\SSO$,
a basis of $\End(E)$, and an ideal class $[\mathfrak c]^2$ such that $(E',\iota') = [\mathfrak c] (E,\iota)$,
one can find the ideal class $[\mathfrak c]$ in probabilistic polynomial time in the length of the input and in\footnote{With the convention that $\max(\emptyset) = +\infty$.} 
$$\min\left(2^{\omega(D)},\max_i\left(\ell_i \mid \ell_i \leq 2^{\omega(D) - i}\right)\right) \ll \min\left(L_{D}(1/2),\#(\cl(\OO)[2]),\ell_{\omega(D)}\right).$$
\end{proposition}

Before proving it, let us recall the following proposition from~\cite{Wes22a}.
\begin{proposition}[\protect{ERH, \cite[Proposition~9]{Wes22a}}]\label{prop:efficientactionfromendring}
Given $(E,\iota) \in \SSO$, a basis of $\End(E)$, and an $\OO$-ideal $\mathfrak a$, one can compute $[\mathfrak a] (E,\iota)$ and an efficient representation of $\varphi_{\mathfrak a} : (E,\iota) \rightarrow [\mathfrak a] (E,\iota)$ in probabilistic polynomial time in the length of the input.
\end{proposition}

\begin{proof}[Proof of Proposition~\ref{prop:findgoodroot}]
Let $B > 0$ be a bound to be tuned later. Consider the sets of prime numbers
\begin{align*}
P_1 &= \{\ell \mid \ell \text{ is an odd prime factor of } \disc(\OO) \text{ and } \ell \leq B\},\text{ and}\\
P_2 &= \{\ell \mid \ell \text{ is an odd prime factor of } \disc(\OO) \text{ and } \ell > B\}.
\end{align*}
For each $\ell \in P_1$, compute $\chi_\ell([\mathfrak c])$ in time $\ell^{O(1)}$ using Theorem~\ref{thm:compl} and the fact that $(E',\iota') = [\mathfrak c] (E,\iota)$. Now, with~\cite{BosS96}, one can compute square roots in $\cl(\OO)$ in polynomial time, so we get an ideal $\mathfrak a$ such that $[\mathfrak a]$ and $[\mathfrak c]$ differ by a two-torsion factor.
From~\cite{BosS96}, one also gets a basis of $\cl(\OO)[2]$, so we can ensure that $\chi_\ell([\mathfrak a]) = \chi_\ell([\mathfrak c])$ for each $\ell \in P_1$.
The solution is now of the form $[\mathfrak c] = [\mathfrak a][\mathfrak b]$ where $[\mathfrak b]$ is in the subgroup $G$ of $\cl(\OO)[2]$ of classes such that $\chi_\ell([\mathfrak b]) = 1$ for all $\ell \in P_1$.
Therefore, the number of remaining candidates for the class $[\mathfrak c]$ is $\#G \leq 2^{\#P_2+1}$.
These can be enumerated (from the basis of $\cl(\OO)[2]$, deduce a basis of the subgroup $G$) and checked for correctness in polynomial time using Proposition~\ref{prop:efficientactionfromendring} and the provided basis of $\End(E)$.
Overall, the running time is polynomial in $\log p$, $\log \disc(\OO)$, $B$, and $2^{\#P_2}$.
The running time follows by choosing $B = \min\left(2^{\omega(D)},\max_i\left(\ell_i \mid \ell_i \leq 2^{\omega(D) - i}\right)\right) $.

Let us prove the last inequality. 
First, $2^{\omega(D)} \ll \#(\cl(\OO)[2])$, so $B \ll \#(\cl(\OO)[2])$. Second, if $\{\ell_i \mid \ell_i \leq 2^{\omega(D) - i}\}$ is empty, then 
$2^{\omega(D) - 1} < \ell_1 \leq \ell_{\omega(D)}$ so $2^{\omega(D)} \ll \ell_{\omega(D)}$.
If it is not empty, clearly $\max_i\left(\ell_i \mid \ell_i \leq 2^{\omega(D) - i}\right) \ll \ell_{\omega(D)}$. In both cases, we deduce $B  \ll \ell_{\omega(D)}$.
Lastly, it remains to see that 
$B \ll L_{D}(1/2).$
Suppose there exists $j$ such that 
$\ell_j = \max_i\left(\ell_i \mid \ell_i \leq 2^{\omega(D) - i}\right)$. 
We have $\log_2(\ell_j) \leq \omega(D) - j$, and
$$\log_2(D) \geq \sum_{i = j+1}^{\omega(D)} \log_2(\ell_i) \geq (\omega(D) - j)\log_2(\ell_j) \geq \log_2(\ell_j)^2.$$
We deduce that $\ell_j \leq 2^{\log_2(D)^{1/2}}$, hence $B \ll L_{D}(1/2)$.
If there exists no such $j$, then
$$\log_2(D) \geq \sum_{i = 1}^{\omega(D)} \log_2(\ell_i) \geq  \sum_{i = 1}^{\omega(D)} (\omega(D) - i) = \Theta(\omega(D)^2),$$
so $2^{\omega(D)} = L_D(1/2)$, hence $B \ll L_{D}(1/2)$.
\end{proof}

The main result of this section is the following theorem.
\begin{theorem}[\protect{ERH, reduction of $\strongOTransform$ to $\OEndRing$}]\label{thm:transformtoendring}
Given an order $\OO$ of discriminant $-D$, the factorization $D = \prod_{i = 1}^{\omega(D)}\ell_i^{e_i}$ (with $\ell_i < \ell_{i+1}$),
three $\OO$-oriented elliptic curves $(E,\iota)$, $(E',\iota')$, $(F,\jmath)\in\SSO$, together with bases of $\End(E)$, $\End(E')$ and $\End(F)$, one can compute (or assert that it does not exist) an $\OO$-ideal $\mathfrak c$ such that $(E',\iota') = [\mathfrak c] (E,\iota)$ and an efficient representation of $\varphi_\mathfrak c : (F,\jmath) \rightarrow [\mathfrak c] (F,\jmath)$
in probabilistic polynomial time in the length of the input and in
$$\min\left(2^{\omega(D)},\max_i\left(\ell_i \mid \ell_i \leq 2^{\omega(D) - i}\right)\right)  \ll \min\left(L_{D}(1/2),\#(\cl(\OO)[2]),\ell_{\omega(D)}\right).$$
\end{theorem}

\begin{remark}
This improves the result of \cite[Thm.\,2]{Wes22a} in two ways. First, the worst case is now sub-exponential: when $D$ is primorial, the running time of \cite[Thm.\,2]{Wes22a} could reach about $D^{1/\log \log D}$, while it is now always at most $L_D(1/2)$. Second, Theorem~\ref{thm:transformtoendring} is now very efficient for a new important family of discriminants: when almost all prime divisors of $D$ are small, no matter how many there are. In particular, primorial numbers (the worst case of \cite[Thm.\,2]{Wes22a}) now benefit from a polynomial time algorithm. 
\end{remark}

\begin{proof}
Thanks to Proposition~\ref{prop:findgoodroot}, the proof is a straightforward adaptation of the proof of~\cite[Thm.\,2]{Wes22a}.
Suppose we are given $(E,\iota),(E',\iota')\in\SSO$, together with $\End(E)$ and $\End(E')$.
Consider the involution $\tau_p : \SSO\rightarrow \SSO$ defined in \cite[Def.\,7]{Wes22a} as $\tau_p(E,\iota) = (E^{(p)},(\phi_p)_*\bar \iota)$, where $\bar\iota$ is the conjugate of $\iota$ (i.e., $\bar\iota(\alpha) = \iota(\overline\alpha)$ for any $\alpha \in \OO$), and $\phi_p : E\rightarrow E^{(p)}$ is the Frobenius isogeny.

Then, per \cite[Prop.\,11]{Wes22a}, one can compute $\mathfrak a$ and $\mathfrak b$ such that $\tau_p(E,\iota) = [\mathfrak a] (E,\iota)$ and $\tau_p(E',\iota') = [\mathfrak b] (E',\iota')$ in polynomial time. From \cite[Lem.\,10]{Wes22a}, the ideal class of $\mathfrak c$ is one of the $\#(\cl(\OO)[2])$ square roots of $[\mathfrak a\overline{\mathfrak b}]$. 
Therefore, the ideal $\mathfrak c$ can be found by Proposition~\ref{prop:findgoodroot} within the claimed running time.
Finally, compute an efficient representation of $\varphi_{\mathfrak c} : (F,\jmath) \rightarrow [\mathfrak c] (F,\jmath)$ in polynomial time with Proposition~\ref{prop:efficientactionfromendring}.
\end{proof}

\begin{corollary}[ERH]\label{coro:OUbertostrongOEndRing}
Given an order $\OO$ of discriminant $-D$, and the factorization $D = \prod_{i = 1}^{\omega(D)}\ell_i^{e_i}$ (with $\ell_i < \ell_{i+1}$),
 $\strongOUber$  reduces to $\strongOEndRing$ in probabilistic polynomial time in the length of the instance and in
 $$\min\left(2^{\omega(D)},\max_i\left(\ell_i \mid \ell_i \leq 2^{\omega(D) - i}\right)\right)  \ll \min\left(L_{D}(1/2),\#(\cl(\OO)[2]),\ell_{\omega(D)}\right).$$
\end{corollary}
\begin{proof}
Again, this is a straightforward adaptation of \cite[Cor.\,4]{Wes22a}.
Suppose we are given $(E,\iota),(F,\jmath) \in \SSO$ and an $\OO$-orientable elliptic curve $E'$. Solving $\strongOEndRing$, one can find $\varepsilon$-bases of $\End(E)$, $\End(F)$ and $\End(E')$, and an $\OO$-orientation $\iota'$ of $E'$. The result follows from Theorem~\ref{thm:transformtoendring}. 
\end{proof}




\noindent \textbf{Conflict of interest statement.}
The authors assert that there are no conflicts of interest.\\

\noindent \textbf{Data availability statement.}
Data sharing is not applicable to this article as no datasets were generated or analysed during the current study.


\bibliography{bib}{}
\bibliographystyle{plain}



\end{document}